\documentclass[reqno]{amsart}
\usepackage{amssymb}
\usepackage{hyperref}
\usepackage[T1]{fontenc}
\usepackage{dsfont}
\newtheorem{thm}{Theorem}[section]

\newtheorem{lem}{Lemma}[section]
\newtheorem{prop}{Proposition}[section]

\begin{document}
\title[ ]{Semi-Fredholm and semi-Browder spectrums for the $\alpha$-times integrated semigroups}
\author[ A. Tajmouati,  A. El Bakkali, M.B. Mohamed Ahmed and H. Boua
\\]
{ A. Tajmouati, A. El Bakkali, M.B. Mohamed Ahmed and H. Boua}
\address{A. Tajmouati, M.B. Mohamed Ahmed and H. Boua\newline
 Sidi Mohamed Ben Abdellah
 Univeristy,
 Faculty of Sciences Dhar Al Mahraz, Fez,  Morocco.}
\email{abdelaziz.tajmouati@usmba.ac.ma}
\email{bbaba2012@gmail.com}
\email{hamid12boua@yahoo.com}
\address{A. El Bakkali  \newline
Department of Mathematics
University Chouaib Doukkali,
Faculty of Sciences.
24000, Eljadida, Morocco.}
\email{aba0101q@yahoo.fr}
\maketitle
\begin{abstract}
In this paper, we describe the different spectrums of the $\alpha$-times integrated semigroups by the spectrums of their generators. Specially, essential ascent and descent, upper and lower semi-Fredholm and semi-Browder spectrums.
\end{abstract}
\section{Introduction}
Let $X$ be a complex Banach space and $\mathcal{B}(X)$ the algebra of all bounded linear operators on  $X$. We denote by $D(T)$, $R(T)$, $R^\infty(T):=\cap_{n\geq 1}R(T^n)$, $N(T)$, $\rho(T)$, $\sigma(T),$ and $\sigma_p(T)$ respectively the domain, the range, the hyper range, the kernel, the resolvent and the spectrum of $T$, where
$\sigma(T)=\{\lambda\in\mathds{C}\ : \lambda-T \,\mbox{is not bijective}\}$ and $\sigma_p(T)=\{\lambda\in\mathds{C}\ : \lambda-T \,\mbox{is not one to one}\}.$ The function resolvent of $T\in\mathcal{B}(X)$ is defined for all $\lambda\in\rho(T)$ by $R(\lambda,T)=(\lambda -T)^{-1}.$ The ascent and descent of an operator $T$ are defined respectively by $$a(T)=\inf\{k\in\mathds{N} : N(T^k)=N(T^{k+1})\}\,\mbox{ and }\,$$
$$d(T)=\inf\{k\in\mathds{N} : R(T^k)=R(T^{k+1})\},$$  with the convention $inf(\varnothing)=\infty$.
The essential ascent and descent of an operator $T$ are defined respectively by $$a_e(T)=\min\{k\in\mathds{N} : \dim N(T^{k+1})/N(T^k)<\infty\}\,\mbox{ and }\,$$
$$d_e(T)=\min\{k\in\mathds{N} : \dim R(T^k)/R(T^{k+1})<\infty\}.$$
The ascent, descent , essential ascent and essential descent spectrums are defined by
  $$\sigma_a(T)=\{\lambda\in\mathds{C}\ : a(\lambda-T)=\infty\};$$
  $$\sigma_d(T)=\{\lambda\in\mathds{C} : d(\lambda-T)=\infty\};$$
  $$\sigma_{a_e}(T)=\{\lambda\in\mathds{C} : a_e(\lambda-T)=\infty\};$$
  $$\sigma_{d_e}(T)=\{\lambda\in\mathds{C} : d_e(\lambda-T)=\infty\};$$
The sets of upper and lower semi-Fredholm and their spectrums are defined respectively by
  $$\Phi_+(X)=\{T\in\mathcal{B}(X) : \delta(T)=\dim N(T)<\infty \,\mbox{and}\,R(T) \,\mbox{is closed}\},$$
  $$\sigma_{e_+}(T)=\{\lambda\in\mathds{C} : \lambda-T\notin \Phi_+(X)\}$$ and
  $$\Phi_-(X)=\{T\in\mathcal{B}(X) : \beta(T)=codim R(T)=\infty\},$$
  $$\sigma_{e_-}(T)=\{\lambda\in\mathds{C}\backslash \lambda-T\notin \Phi_-(X)\}.$$
  An operator $T\in \mathcal{B}(X)$ is called semi-Fredholm, in symbol $T\in \Phi_\pm(X)$, if $$T\in \Phi_+(X)\cup\Phi_-(X).$$
  An operator $T\in \mathcal{B}(X)$ is called Fredholm, in symbol $T\in \Phi(X)$, if $$T\in \Phi_+(X)\cap\Phi_-(X).$$
  The essential and semi-Fredholm spectrums are defined by
  $$\sigma_e(T)=\{\lambda\in\mathds{C} : \lambda-T\notin \Phi(X)\};$$
  $$\sigma_{e_\pm}(T)=\{\lambda\in\mathds{C} : \lambda-T\notin \Phi_\pm(X)\};$$
The sets of upper and lower semi-Browder and their spectrums are defined respectively by
$${Br}_+(X)=\{T\in\Phi_+(X) : a(T)<\infty\},$$
  $$\sigma_{{Br}_+}(T)=\{\lambda\in\mathds{C} : \lambda-T\notin Br_+(X)\}$$ and
  $${Br}_-(X)=\{T\in\Phi_-(X) : d(T)<\infty\},$$
  $$\sigma_{Br_-}(T)=\{\lambda\in\mathds{C} : \lambda-T\notin Br_-(X)\}.$$
  An operator $T\in \mathcal{B}(X)$ is called semi-Browder, in symbol $T\in Br_\pm(X)$, if
  $$T\in {Br}_+(X)\cup {Br}_-(X).$$
  An operator $T\in \mathcal{B}(X)$ is called Browder, in symbol $T\in Br(X)$, if
  $$T\in Br_+(X)\cap Br_-(X).$$
The semi-Browder and Browder spectrums are defined by
$$\sigma_{{Br}_{\pm}}(T)=\{\lambda\in\mathds{C} : \lambda-T\notin {Br}_\pm(X)\}$$
and
$$\sigma_{Br}(T)=\{\lambda\in\mathds{C} : \lambda-T\notin Br(X)\}.$$
Let $\alpha\geq 0$ and let $A$ be a linear operator on a Banach space $X$. We recall that $A$ is the generator of an $\alpha$-times integrated semigroup
$(S(t))_{t\geq 0}$ on $X$ \cite{r.5} if
$]\omega, +\infty[ \subseteq \rho(A)$ for some $\omega\in \mathds{R}$
and there exists a strongly continuous mapping $S: [0, +\infty[ \rightarrow \mathcal{B}(X)$ satisfying
\begin{eqnarray*}
\|S(t)\|&\leq& Me^{\omega t} \,\,\mbox{for all}\,\, t\geq 0 \,\,\mbox{ and some }\,\, M > 0;\\
R(\lambda, A)&=& \lambda^\alpha\int_0^{+\infty} e^{-\lambda t}S(t)ds \,\,\mbox{for all}\,\, \lambda> \max\{\omega, 0\},
\end{eqnarray*}
in this case, $(S(t))_{t\geq 0}$ is called an $\alpha$-times integrated semigroup and the domain of its generator $A$ is defined by
$$D(A)=\left\{x\in X : \int_0^tS(s)Axds=S(t)x-\frac{t^\alpha x}{\Gamma(\alpha+1)}\right\},$$
where $\Gamma$ is the Euler integral giving by $$\Gamma(\alpha+1)=\int_0^{+\infty} x^\alpha e^{-x}dx.$$
We know that $(S(t))_{t\geq 0}\subseteq\mathcal{B}(X)$ is an $\alpha$-times integrated semigroup if and only if
$$S(t+s)=\frac{1}{\Gamma(\alpha)}\left(\int_t^{t+s}(t+s-r)^{\alpha-1}S(r)xdr-
\int_0^{s}(t+s-r)^{\alpha-1}S(r)xdr\right)$$ for all $x\in X$ and all $t,s\geq 0$.\\
In \cite{r.3}, the authors have studied the different spectrums of the 1-times integrated semigroups.
In our paper \cite{r.12}, we have studied descent, ascent, Drazin, Fredholm and Browder spectrums of the $\alpha$-times integrated semigroups. In this paper, we continue to study $\alpha$-times integrated semigroups for all $\alpha\geq 0$. We investigate the relationships between the different spectrums of an $\alpha$-times integrated semigroup and their generators, precisely the essential ascent and descent, upper and lower semi-Fredholm and semi-Browder spectrums.

\section{Main results}

\begin{lem}\label{l0}\cite[Proposition 2.4]{r.1}
Let $A$ be the generator of an $\alpha$-times integrated semigroup $(S(t))_{t\geq 0}\subseteq \mathcal{B}(X)$ where $\alpha \geq 0.$ Then for all $x\in D(A)$ and all $t\geq 0$ we have
\begin{enumerate}
  \item $S(t)x\in D(A)$ and $AS(t)x=S(t)Ax.$
  \item $S(t)x=\frac{t^\alpha}{\Gamma(\alpha+1)}x+\int_0^t S(s)Axds.$
\end{enumerate}
Moreover, for all $x\in X$ we get $\int_0^t S(s)xds\in D(A)$ and
  $$A\int_0^t S(s)xds=S(t)x -\frac{t^\alpha}{\Gamma(\alpha+1)}x.$$
\end{lem}
We begin by the lemmas.
\begin{lem}\label{l1} Let $A$ be the generator of an $\alpha$-times integrated semigroup $(S(t))_{t\geq 0}$ with $\alpha>0$. Then for all $\lambda\in\mathds{C}$ and all $t\geq 0$
\begin{enumerate}
\item $(\lambda-A)D_\lambda(t)x=\int_0^t e^{\lambda(t-s)}\frac{s^{\alpha-1} x}{\Gamma(\alpha)}ds-S(t)x,\,\,\forall x\in X$  where $$D_\lambda(t)x=\int_0^t e^{\lambda (t-r)}S(r)dr;$$
\item $D_\lambda(t)(\lambda-A)x=\int_0^t e^{\lambda (t-s)}\frac{s^{\alpha-1} x}{\Gamma(\alpha)}ds-S(t)x,\,\,\forall x\in D(A)$.
\end{enumerate}
\end{lem}
\begin{proof}
\begin{enumerate}
 \item By Lemma \ref{l0}, we know that for all $x\in D(A)$
 $$S(s)x=\frac{s^\alpha}{\Gamma(\alpha+1)}x+\int_0^sS(r)Axdr.$$
 Then, since $\Gamma(\alpha+1)=\alpha\Gamma(\alpha)$, we obtain
 $$S'(s)x= \frac{s^{\alpha-1}}{\Gamma(\alpha)}x+S(s)Ax.$$
 Therefore, we conclude that
 \begin{eqnarray*}
 D_\lambda(t)Ax &=&\int_0^te^{\lambda(t-s)}S(s)Axds\\
  &=& \int_0^te^{\lambda(t-s)}[S'(s)x-\frac{s^{\alpha-1}}{\Gamma(\alpha)}x]ds\\
 &=& \int_0^te^{\lambda(t-s)}S'(s)xds-
 \int_0^te^{\lambda(t-s)}\frac{s^{\alpha-1}}{\Gamma(\alpha)}xds\\
 &=& S(t)x+\lambda D_\lambda(t)x-
   \int_0^te^{\lambda(t-s)}\frac{s^{\alpha-1}}{\Gamma(\alpha)}xds
     \end{eqnarray*}
Finally, we obtain for all $x\in D(A)$
$$D_\lambda(t)(\lambda-A)x= \left(\int_0^t e^{\lambda(t-s)}\frac{s^{\alpha-1}}{\Gamma(\alpha)}ds - S(t)\right)x.$$
\item  Let $\mu\in \rho(A)$. From proof of Lemma \ref{l0}, we have
for all $x\in X$ $$R(\mu,A)S(s)x=S(s)R(\mu,A)x.$$
Hence, for all $x\in X$ we conclude
\begin{eqnarray*}
R(\mu,A)D_\lambda(t)x &=& R(\mu,A)\int_0^te^{\lambda (t-s)}S(s)xds\\
   &=& \int_0^te^{\lambda (t-s)}R(\mu,A)S(s)xds\\
   &=& \int_0^te^{\lambda (t-s)}S(s)R(\mu,A)xds\\
&=& D_\lambda(t)R(\mu,A)x.
   \end{eqnarray*}
Therefore, we obtain for all $x\in X$
 \begin{eqnarray*}
 D_\lambda(t)x &=& \int_0^te^{\lambda (t-s)}S(s)xds\\
&=& \int_0^te^{\lambda (t-s)}S(s)(\mu-A)R(\mu,A)xds  \\
&=& \mu \int_0^te^{\lambda (t-s)} S(s)R(\mu,A)xds -\int_0^te^{\lambda (t-s)}S(s)AR(\mu,A)xds \\
&=& \mu \int_0^te^{\lambda (t-s)} R(\mu,A)S(s)xds -\int_0^te^{\lambda (t-s)}S(s)AR(\mu,A)xds \\
&=& \mu R(\mu,A)\int_0^te^{\lambda (t-s)} S(s)xds -\int_0^te^{\lambda (t-s)}S(s)AR(\mu,A)xds \\
&=& \mu R(\mu,A)D_\lambda(t)x-D_\lambda(t)AR(\mu,A)x\\
&=& \mu R(\mu,A)D_\lambda(t)x-\left(S(t)R(\mu,A)x+\lambda D_\lambda(t)R(\mu,A)x-\int_0^te^{\lambda(t-s)}\frac{s^{\alpha-1}}{\Gamma(\alpha)}R(\mu,A)xds\right)\\
&=& \mu R(\mu,A)D_\lambda(t)x-\left(R(\mu,A)S(t)x+\lambda R(\mu,A)D_\lambda(t)x-R(\mu,A)\int_0^te^{\lambda(t-s)}\frac{s^{\alpha-1}}{\Gamma(\alpha)}xds\right)\\ &=&  R(\mu,A)\big[(\mu-\lambda)D_\lambda(t)x-S(t)x+
\int_0^te^{\lambda(t-s)}\frac{s^{\alpha-1}}{\Gamma(\alpha)}xds\big]
\end{eqnarray*}
Therefore, for all $x\in X$ we have $D_\lambda(t)x\in D(A)$ and
$$(\mu-A)D_\lambda(t)x=(\mu-\lambda) D_\lambda(t)x+\int_0^t e^{\lambda(t-s)}\frac{s^{\alpha-1}}{\Gamma(\alpha)}xds -
S(t)x.$$
Finally, for all $x\in X$ and all $\lambda\in \mathds{C}$ we obtain
$$(\lambda-A)D_\lambda(t)x=\left(\int_0^t e^{\lambda(t-s)}\frac{s^{\alpha-1}}{\Gamma(\alpha)}ds -
S(t)\right)x.$$
\end{enumerate}
\end{proof}

\begin{lem}\label{l2} Let $A$ be the generator of an $\alpha$-times integrated semigroup $(S(t))_{t\geq 0}$ with $\alpha>0$. Then for all $\lambda\in\mathds{C}$, all $t\geq 0$ and all $x\in X$
\begin{enumerate}
\item We have the identity  $$\,\,(\lambda-A)L_\lambda(t)+\varphi_\lambda(t)D_\lambda(t)=\phi_\lambda(t)I,$$
    where $L_\lambda(t)=\int_0^t e^{-\lambda s }D_\lambda(s)ds,\, \varphi_\lambda(t)=e^{\lambda t}\,\mbox{and}\, \phi_\lambda(t)=\int_0^t\int_0^\tau e^{-\lambda r}\frac{r^{\alpha-1}}{\Gamma(\alpha)}drd\tau.$\\
Moreover, the operator $L_\lambda(t)$ is commute with each one of $D_\lambda(t)$and $(\lambda-A)$.
\item For all $n\in\mathds{N}^*,$ there exists an $L_{\lambda,n}(t)\in \mathcal{B}(X)$ such that $$(\lambda-A)L_{\lambda,n}(t)+[\varphi_\lambda(t)]^n[D_\lambda(t)]^n=
    [\phi_\lambda(t)]^nI.$$
    Moreover, the operator $L_{\lambda,n}(t)$ is commute with each one of $D_\lambda(t)$ and $\lambda-A$.
\item For all $n\in\mathds{N}^*,$ there exists an operator $D_{\lambda,n}(t)\in \mathcal{B}(X)$ such that $$(\lambda-A)^n[L_\lambda(t)]^n+D_{\lambda,n}(t)D_\lambda(t)=[\phi_\lambda(t)]^nI.$$
    Moreover, the operator $D_{\lambda,n}(t)$ is commute with each one of $D_\lambda(t)$, $L_\lambda(t)$ and $\lambda-A$.
 \item  For all $n\in\mathds{N}^*,$ there exists an operator $K_{\lambda,n}(t)\in \mathcal{B}(X)$ such that
$$ (\lambda-A)^nK_{\lambda,n}(t)+[D_{\lambda,n}(t)]^n [D_\lambda(t)]^n=[\phi_\lambda(t)]^{n^2}I,$$
Moreover, the operator $K_{\lambda,n}(t)$ is commute with each one of $D_\lambda(t)$, $D_{\lambda,n}(t)$ and $\lambda-A$.
\end{enumerate}
\end{lem}
\begin{proof}
\begin{enumerate}
\item Let $\mu\in \rho(A)$. By Lemma \ref{l1}, for all $x\in X$ we have $D_\lambda(s)x\in D(A)$ and hence
 \begin{eqnarray*}
 L_\lambda(t)x &=& \int_0^te^{-\lambda s}D_\lambda(s)xds\\
 &=& \int_0^te^{-\lambda s}R(\mu,A)(\mu-A)D_\lambda(s)xds\\
 &=& R(\mu,A)[\mu\int_0^te^{-\lambda s}D_\lambda(s)xds-\int_0^te^{-\lambda s}AD_\lambda(s)xds]\\
  &=& R(\mu,A)[\mu L_\lambda(t)x-\int_0^te^{-\lambda s}AD_\lambda(s)xds]
\end{eqnarray*}
Therefore for all $x\in X$, we have $L_\lambda(t)x\in D(A)$ and
$$(\mu-A)L_\lambda(t)x=\mu L_\lambda(t)x-\int_0^te^{-\lambda s}AD_\lambda(s)xds.$$
Thus
$$AL_\lambda(t)x=\int_0^te^{-\lambda s}AD_\lambda(s)xds.$$
Hence, we conclude that
\begin{eqnarray*}
(\lambda-A)L_\lambda(t)x&=&\lambda L_\lambda(t)x -\int_0^te^{-\lambda s}AD_\lambda(s)xds\\
&=&\lambda L_\lambda(t)x -\int_0^te^{-\lambda s}\big[\lambda D_\lambda(s)x-\int_0^s e^{\lambda(s-r)}\frac{r^{\alpha-1}}{\Gamma(\alpha)}xdr +
S(s)x\big]ds\\
&=&\lambda L_\lambda(t)x -\lambda\int_0^te^{-\lambda s}D_\lambda(s)x ds +\int_0^te^{-\lambda s}\int_0^s e^{\lambda(s-r)}\frac{r^{\alpha-1}}{\Gamma(\alpha)}xdrds -\int_0^te^{-\lambda s}
S(s)xds\\
&=&\lambda L_\lambda(t)x -\lambda L_\lambda(t)x +\int_0^t\int_0^s e^{-\lambda r}\frac{r^{\alpha-1}}{\Gamma(\alpha)}xdrds -e^{-\lambda t}\int_0^te^{\lambda (t-s)}S(s)xds  \\
&=&\int_0^t\int_0^s e^{-\lambda r}\frac{r^{\alpha-1}}{\Gamma(\alpha)}xdrds -e^{-\lambda t}D_\lambda (t)x  \\
 &=& \big[\phi_\lambda(t)I-\varphi_\lambda(t)D_\lambda(t)\big]x,
\end{eqnarray*}
where $\phi_\lambda(t)=\int_0^t\int_0^s e^{-\lambda r}\frac{r^{\alpha-1}}{\Gamma(\alpha)}drds$ and $\varphi_\lambda(t)=e^{-\lambda t}.$\\
Therefore, we obtain $$(\lambda-A)L_\lambda(t)+\varphi_\lambda(t)D_\lambda(t)=\phi_\lambda(t)I.$$
Since $S(s)S(t)=S(t)S(s)$ for all $s,t\geq 0$, then
 $D_\lambda(s)S(t)=S(t)D_\lambda(s).$\\
Hence
\begin{eqnarray*}
D_\lambda(t)D_\lambda(s) &=& \int_0^te^{\lambda(t-r)}S(r)D_\lambda(s)dr\\
&=& \int_0^te^{\lambda(t-r)}S(r)D_\lambda(s)dr\\
&=& \int_0^te^{\lambda(t-r)}D_\lambda(s)S(r)dr\\
&=& D_\lambda(s)\int_0^te^{\lambda(t-r)}S(r)dr\\
&=& D_\lambda(s)D_\lambda(t).
\end{eqnarray*}
Thus, we deduce that
\begin{eqnarray*}
D_\lambda(t)L_\lambda(t) &=& D_\lambda(t)\int_0^te^{-\lambda s}D_\lambda(s)ds\\
&=& \int_0^te^{-\lambda s}D_\lambda(t)D_\lambda(s)ds\\
&=& \int_0^te^{-\lambda s}D_\lambda(s)D_\lambda(t)ds\\
&=& \int_0^te^{-\lambda s}D_\lambda(s)dsD_\lambda(t)\\
&=& L_\lambda(t)D_\lambda(t).
\end{eqnarray*}
Since for all $x\in X$ $AL_\lambda(t)x=\int_0^te^{-\lambda s}AD_\lambda(s)xds$ and for all $x\in D(A)$ $AD_\lambda(s)x=D_\lambda(s)Ax,$
then we obtain for all $x\in D(A)$
\begin{eqnarray*}
(\lambda-A)L_\lambda(t)x &=& \lambda L_\lambda(t)x-AL_\lambda(t)x\\
&=& \lambda L_\lambda(t)x -\int_0^te^{-\lambda s}AD_\lambda(s)xds\\
&=& \lambda L_\lambda(t)x -\int_0^te^{-\lambda s}AD_\lambda(s)xds\\
&=&\lambda L_\lambda(t)x -\int_0^te^{-\lambda s}D_\lambda(s)Axds\\
&=& \lambda L_\lambda(t)x- L_\lambda(t)Ax\\
&=& L_\lambda(t)(\lambda-A)x.
\end{eqnarray*}
\item Since $(\lambda-A)L_\lambda(t)+\varphi_\lambda(t)D_\lambda(t)=\phi_\lambda(t)I$, then for all $n\in \mathds{N}^*$ we obtain
\begin{eqnarray*}
[\varphi_\lambda(t)D_\lambda(t)]^n &=&[\phi_\lambda(t)I-(\lambda-A)L_\lambda(t)]^n\\
&=&\sum_{i=0}^n C_n^i[\phi_\lambda(t)]^{n-i}[-(\lambda-A)L_\lambda(t)]^i\\
&=& [\phi_\lambda(t)]^nI -(\lambda-A)\sum_{i=1}^n C_n^i[\phi_\lambda(t)]^{n-i}[-(\lambda-A)]^{i-1}[L_\lambda(t)]^i\\
&=& [\phi_\lambda(t)]^nI -(\lambda-A)L_{\lambda,n}(t),
\end{eqnarray*}
where $$L_{\lambda,n}(t)=\sum_{i=1}^nC_n^i [\phi_\lambda(t)]^{n-i}[-(\lambda-A)]^{i-1}[L_\lambda(t)]^i.$$
Therefore, we have
$$(\lambda-A)L_{\lambda,n}(t)+[\varphi_\lambda(t)]^n[D_\lambda(t)]^n =[\phi_\lambda(t)]^nI.$$
Finally, for commutativity, it is clear that $L_{\lambda,n}(t)$ commute with each one of $D_\lambda(t)$ and $\lambda-A$.
\item For all $n\in \mathds{N}^*$, we obtain
\begin{eqnarray*}
[(\lambda-A)L_\lambda(t)]^n &=&[\phi_\lambda(t)I-\varphi_\lambda(t)D_\lambda(t)]^n\\
&=&\sum_{i=0}^n C_n^i[\phi_\lambda(t)]^{n-i}[-\varphi_\lambda(t)D_\lambda(t)]^i\\
&=& [\phi_\lambda(t)]^nI -D_\lambda(t)\sum_{i=1}^nC_n^i [\phi_\lambda(t)]^{n-i}[\varphi_\lambda(t)]^{i}[-D_\lambda(t)]^{i-1}\\
&=& [\phi_\lambda(t)]^nI -D_\lambda(t)D_{\lambda,n}(t),
\end{eqnarray*}
where $$D_{\lambda,n}(t)=\sum_{i=1}^n C_n^i [\phi_\lambda(t)]^{n-i}[\varphi_\lambda(t)]^{i}[-D_\lambda(t)]^{i-1}.$$
Therefore, we have
$$(\lambda-A)^n[L_\lambda(t)]^n+D_\lambda(t)D_{\lambda,n}(t)=[\phi_\lambda(t)]^nI.$$
Finally, for commutativity, it is clear that $D_{\lambda,n}(t)$ commute with each one of $D_\lambda(t)$, $L_\lambda(t)$ and $\lambda-A$.
\item
Since we have
$D_\lambda(t)D_{\lambda,n}(t)=[\phi_\lambda(t)]^nI-(\lambda-A)^n[L_\lambda(t)]^n,$
then for all $n\in \mathds{N}$
\begin{eqnarray*}
[D_\lambda(t)D_{\lambda,n}(t)]^n &=& \big[[\phi_\lambda(t)]^nI-(\lambda-A)^n[L_\lambda(t)]^n\big]^n\\
 &=& [\phi_\lambda(t)]^{n^2}I-\sum_{i=1}^n C_n^i \big[[\phi_\lambda(t)]^{n}\big]^{n-i}\big[(\lambda-A)^n[L_\lambda(t)]^n\big]^i\\
 &=& [\phi_\lambda(t)]^{n^2}I-(\lambda-A)^n\sum_{i=1}^nC_n^i [\phi_\lambda(t)]^{n(n-i)}(\lambda-A)^{n(i-1)}[L_\lambda(t)]^{ni}\\
&=& [\phi_\lambda(t)]^{n^2}I-(\lambda-A)^nK_{\lambda,n}(t),
\end{eqnarray*}
where $K_{\lambda,n}(t)=\sum_{i=1}^nC_n^i
[\phi_\lambda(t)]^{n(n-i)}(\lambda-A)^{n(i-1)}[L_\lambda(t)]^{ni}.$
Hence we obtain
$$[D_\lambda(t)]^n[D_{\lambda,n}(t)]^n +(\lambda-A)^nK_{\lambda,n}(t) =[\phi_\lambda(t)]^{n^2}I.$$
Finally, the commutativity is clear.
\end{enumerate}
\end{proof}

We start by this result.
\begin{prop}\label{p1} Let $A$ be the generator of an $\alpha$-times integrated semigroup $(S(t))_{t\geq 0}$ with $\alpha>0$. For all $\lambda\in\mathds{C}$ and all $t\geq 0$, if $R\left(\int_0^te^{\lambda (t-s)}\frac{s^{\alpha-1}}{\Gamma(\alpha)}ds-S(t)\right)$ is closed, then $\lambda-A$ is also closed.
\end{prop}

\begin{proof}
Let $(y_n)_{n\in\mathds{N}}\subseteq X$ such that $y_n\rightarrow y\in X$ and there exists $(x_n)_{n\in\mathds{N}}\subseteq D(A)$ satisfying
      $$(\lambda-A)x_n=y_n.$$
      By Lemma \ref{l2}, we obtain
      $$(\lambda-A)L_\lambda(t)y_n+G_\lambda(t)D_\lambda(t)y_n=\phi_\lambda(t)y_n.$$
Hence, we conclude that
\begin{eqnarray*}
\left(\int_0^te^{\lambda (t-s)}\frac{s^{\alpha-1}}{\Gamma(\alpha)}ds-S(t)\right)G_\lambda(t)x_n
&=& D_\lambda(t)(\lambda -A)G_\lambda(t)x_n\\
&=& G_\lambda(t)D_\lambda(t)(\lambda -A)x_n\\
&=& G_\lambda(t)D_\lambda(t)y_n\\
&=& \phi_\lambda(t)y_n-(\lambda-A)L_\lambda(t)y_n.
\end{eqnarray*}
Thus, $$\phi_\lambda(t)y_n-(\lambda-A)L_\lambda(t)y_n \in R\left(\int_0^te^{\lambda (t-s)}\frac{s^{\alpha-1}}{\Gamma(\alpha)}ds-S(t)\right).$$
Therefore, since $R\left(\int_0^te^{\lambda (t-s)}\frac{s^{\alpha-1}}{\Gamma(\alpha)}ds-S(t)\right)$ is closed, $L_\lambda(t)$ is bounded linear and
$\phi_\lambda(t)y_n-(\lambda-A)L_\lambda(t)y_n$ converges to  $\phi_\lambda(t)y-(\lambda-A)L_\lambda(t)y,$ we conclude that
$$\phi_\lambda(t)y-(\lambda-A)L_\lambda(t)y\in R\left(\int_0^te^{\lambda (t-s)}\frac{s^{\alpha-1}}{\Gamma(\alpha)}ds-S(t)\right).$$
Then there exists $z\in X$ such that
$$\left(\int_0^te^{\lambda (t-s)}\frac{s^{\alpha-1}}{\Gamma(\alpha)}ds-S(t)\right)z=\phi_\lambda(t)y-(\lambda-A)L_\lambda(t)y.$$
Hence for all $t\neq 0$, we have $\phi_\lambda(t)\neq 0$ and
\begin{eqnarray*}
y &=& \frac{1}{\phi_\lambda(t)}\left([\int_0^te^{\lambda (t-s)}\frac{s^{\alpha-1}}{\Gamma(\alpha)}ds-S(t)]z+(\lambda-A)L_\lambda(t)y\right);\\
&=& \frac{1}{\phi_\lambda(t)}\left((\lambda-A)D_\lambda(t)z+(\lambda-A)L_\lambda(t)y\right);\\
&=& \frac{1}{\phi_\lambda(t)}(\lambda-A)[D_\lambda(t)z+L_\lambda(t)y].\\
\end{eqnarray*}
Finally, we obtain $$y\in R(\lambda-A).$$
\end{proof}

The following result discusses the semi-Fredholm spectrum.
\begin{thm}\label{t1} Let $A$ be the generator of an $\alpha$-times integrated semigroup $(S(t))_{t\geq 0}$ with $\alpha>0$. Then for all $t\geq 0$
\begin{enumerate}
\item $\int_0^t e^{(t-s)\sigma_{e_+}(A)}\frac{s^{\alpha-1}}{\Gamma(\alpha)}ds\subseteq \sigma_{e_+}(S(t));$
\item $\int_0^t e^{(t-s)\sigma_{e_-}(A)}\frac{s^{\alpha-1}}{\Gamma(\alpha)}ds\subseteq \sigma_{e_-}(S(t));$
\item $\int_0^t e^{(t-s)\sigma_{e_\pm}(A)}\frac{s^{\alpha-1}}{\Gamma(\alpha)}ds\subseteq
 \sigma_{e_\pm}(S(t)).$
\end{enumerate}
\end{thm}

\begin{proof}
\begin{enumerate}
\item Suppose that $\int_0^te^{\lambda(t-s)}\frac{s^{\alpha-1}}{\Gamma(\alpha)}ds\notin\sigma_{e_+}(S(t))$, then there exists $n\in\mathds{N}$ such that
$\delta\left(\int_0^t e^{\lambda(t-s)}\frac{s^{\alpha-1}}{\Gamma(\alpha)}ds-S(t)\right)=n$ and
$R\left(\int_0^t e^{\lambda(t-s)}\frac{s^{\alpha-1}}{\Gamma(\alpha)}ds-S(t)\right)$ is closed.\\
By Lemma \ref{l1}, we obtain $$N(\lambda-A)\subset N\left(\int_0^t e^{\lambda(t-s)}\frac{s^{\alpha-1}}{\Gamma(\alpha)}ds-S(t)\right),$$
 then $$\delta(\lambda-A)\leq n.$$
On the other hand, from Proposition \ref{p1}, we deduce that $R(\lambda-A)$ is closed.
Therefore $$\lambda\notin \sigma_{e_+}(A).$$
\item Suppose that $\int_0^t e^{\lambda(t-s)}\frac{s^{\alpha-1}}{\Gamma(\alpha)}ds\notin\sigma_{e_-}(S(t))$, then there exist $n\in\mathds{N}$ such that
$\beta\left(\int_0^t e^{\lambda(t-s)}\frac{s^{\alpha-1}}{\Gamma(\alpha)}ds-S(t)\right)=n$.\\
By Lemma \ref{l1}, we obtain $$ R\left(\int_0^t e^{\lambda(t-s)}\frac{s^{\alpha-1}}{\Gamma(\alpha)}ds-S(t)\right)\subseteq R(\lambda-A),$$
 then $\beta(\lambda-A)\leq n$ and hence $$\lambda\notin\sigma_{e_-}(A).$$
\item It is automatic by the previous assertions of this theorem.
\end{enumerate}
\end{proof}

\begin{prop}\label{p2} Let $A$ be the generator of an $\alpha$-times integrated semigroup $(S(t))_{t\geq 0}$ and $\alpha>0$. Then for all $\lambda\in\mathds{C}$ and all $t\geq 0$, we have
\begin{enumerate}
\item $d\left(\int_0^te^{\lambda (t-s)}\frac{s^{\alpha+1}}{\Gamma(\alpha+)}ds-S(t)\right)=n,$ then  $d(\lambda-A)\leq n.$
\item $a\left(\int_0^te^{\lambda (t-s)}\frac{s^{\alpha+1}}{\Gamma(\alpha)}ds-S(t)\right)=n,$ then  $a(\lambda-A)\leq n.$
\end{enumerate}
\end{prop}
\begin{proof}
$\,$\\
\begin{enumerate}
  \item Let $y\in R(\lambda-A)^n$, then there exists $x\in D(A^n)$ satisfying $$(\lambda-A)^nx=y.$$
       Since $d\left(\int_0^te^{\lambda (t-s)}\frac{s^{\alpha+1}}{\Gamma(\alpha)}ds-S(t)\right)=n,$ therefore
      $$R\left(\int_0^te^{\lambda (t-s)}\frac{s^{\alpha+1}}{\Gamma(\alpha)}ds-S(t)\right)^n=R\left(\int_0^te^{\lambda (t-s)}\frac{s^{\alpha+1}}{\Gamma(\alpha)}ds-S(t)\right)^{n+1}.$$
      Hence there exists $z\in X$ such that
     $$ \left(\int_0^te^{\lambda (t-s)}\frac{s^{\alpha+1}}{\Gamma(\alpha)}ds-S(t)\right)^{nx}=\left(\int_0^te^{\lambda (t-s)}\frac{s^{\alpha+1}}{\Gamma(\alpha)}ds-S(t)\right)^{n+1}z.$$
On the other hand, by Lemma \ref{l2}, we have $$(\lambda-A)L_{\lambda,n}(t)+[\varphi_\lambda(t)]^n[D_\lambda(t)]^n=[\phi_\lambda(t)]^nI,$$
     with $L_{\lambda,n}(t)$, $D_\lambda(t)$ and $(\lambda-A)$ are pairwise commute.\\
Thus, we have
     \begin{eqnarray*}
     [\phi_\lambda(t)]^ny&=&(\lambda-A)^n[\phi_\lambda(t)]^nx\\
     &=& (\lambda-A)^n\left(\lambda-A)L_{\lambda,n}(t)+[\varphi_\lambda(t)\right)^n[D_\lambda(t)]^n]x\\
     &=&(\lambda-A)^n(\lambda-A)L_{\lambda,n}(t)x+[\varphi_\lambda(t)]^n
     (\lambda-A)^n[D_\lambda(t)^n]x\\
     &=&(\lambda-A)^{n+1}L_{\lambda,n}(t)x+[\varphi_\lambda(t)]^n\left(\int_0^te^{\lambda (t-s)}\frac{s^{\alpha-1}}{\Gamma(\alpha)}ds-S(t)\right)^nx\\
     &=&(\lambda-A)^{n+1}L_{\lambda,n}(t)x+[\varphi_\lambda(t)]^n\left(\int_0^te^{\lambda (t-s)}\frac{s^{\alpha-1}}{\Gamma(\alpha)}ds-S(t)\right)^{n+1}z\\
     &=&(\lambda-A)^{n+1}L_{\lambda,n}(t)x+[\varphi_\lambda(t)]^n\left((\lambda-A)^{n+1}
     [D_\lambda(t)]^{n+1}z\right)\\
     &=&(\lambda-A)^{n+1}\left(L_{\lambda,n}(t)x+[\varphi_\lambda(t)]^n[D_\lambda(t)]^{n+1}z\right).
     \end{eqnarray*}
Since $\phi_\lambda(t)\neq 0$ for $t>0$, we conclude that $y\in R(\lambda-A)^{n+1}$ and hence  $$R(\lambda-A)^n=R(\lambda-A)^{n+1}.$$
Finally, we conclude that $$d(\lambda-A)\leq n.$$
  \item Let $x\in N(\lambda-A)^{n+1}$ and we suppose that $a\left(\int_0^te^{\lambda (t-s)}\frac{s^{\alpha-1}}{\Gamma(\alpha)}ds-S(t)\right)=n$, then we obtain
$$N\left(\int_0^te^{\lambda(t-s)}\frac{s^{\alpha-1}}{\Gamma(\alpha)}ds-S(t)\right)^n=
N\left(\int_0^te^{\lambda (t-s)}\frac{s^{\alpha-1}}{\Gamma(\alpha)}ds-S(t)\right)^{n+1}.$$
From Lemma \ref{l1}, we deduce that
$$N(\lambda-A)^{n+1}\subseteq N\left(\int_0^te^{\lambda (t-s)}\frac{s^{\alpha-1}}{\Gamma(\alpha)}ds-S(t)\right)^{n+1},$$
hence $$x\in N\left(\int_0^te^{\lambda (t-s)}\frac{s^{\alpha-1}}{\Gamma(\alpha)}ds-S(t)\right)^n.$$
Thus, we have
   \begin{eqnarray*}
   {[\phi_\lambda(t)]^n}(\lambda-A)^nx &=& (\lambda-A)^n[(\lambda-A)L_{\lambda,n}(t)+[\varphi_\lambda(t)]^n[D_\lambda(t)]^n]x;\\
     &=&(\lambda-A)^n(\lambda-A)L_{\lambda,n}(t)x+
     [\varphi_\lambda(t)]^n(\lambda-A)^n[D_\lambda(t)]^nx\\
     &=&(\lambda-A)^{n+1}L_{\lambda,n}(t)x+[\varphi_\lambda(t)]^n\left(\int_0^te^{\lambda (t-s)}\frac{s^{\alpha-1}}{\Gamma(\alpha)}ds-S(t)\right)^nx\\
     &=&(\lambda-A)^{n+1}L_{\lambda,n}(t)x\\
     &=&L_{\lambda,n}(t)(\lambda-A)^{n+1}x\\
     &=& 0.
    \end{eqnarray*}
    Therefore, since $\phi_\lambda(t)\neq 0$ for $t>0$, we obtain $x\in N(\lambda-A)^n$ and hence $$a(\lambda-A)\leq n.$$
\end{enumerate}
\end{proof}

The following theorem examines the semi-Browder spectrum.
\begin{thm} Let $A$ be the generator of an $\alpha$-times integrated semigroup $(S(t))_{t\geq 0}$ with $\alpha>0$. Then for all $t\geq 0$
\begin{enumerate}
\item $\int_0^t e^{(t-s)\sigma_{{Br}_+}(A)}\frac{s^{\alpha-1}}{\Gamma(\alpha)}ds\subseteq \sigma_{{Br}_+}(S(t));$
\item $\int_0^t e^{(t-s)\sigma_{{Br}_-}(A)}\frac{s^{\alpha-1}}{\Gamma(\alpha)}ds\subseteq \sigma_{{Br}_-}(S(t));$
\item $\int_0^t e^{(t-s)\sigma_{{Br}_\pm}(A)}\frac{s^{\alpha-1}}{\Gamma(\alpha)}ds\subseteq \sigma_{{Br}_\pm}(S(t)).$
 \end{enumerate}
\end{thm}

\begin{proof}
\begin{enumerate}
  \item Suppose that $\int_0^t e^{\lambda(t-s)}\frac{s^{\alpha-1}}{\Gamma(\alpha)}ds\notin\sigma_{Br_+}(S(t))$, then there exist $n,m\in\mathds{N}$ such that
$\delta\left(\int_0^t e^{\lambda(t-s)}\frac{s^{\alpha-1}}{\Gamma(\alpha)}ds-S(t)\right)=m$,
$R\left(\int_0^t e^{\lambda(t-s)}\frac{s^{\alpha-1}}{\Gamma(\alpha)}ds-S(t)\right)$ is closed and $a\left(\int_0^t e^{\lambda(t-s)}\frac{s^{\alpha-1}}{\Gamma(\alpha)}ds-S(t)\right)=n.$
From Lemma \ref{l1} and Propositions \ref{p1} and \ref{p2}, we obtain\\
$\delta(\lambda-A)\leq m$, $R(\lambda-A)$ is closed and $a(\lambda-A)\leq n.$\\
Therefore $\lambda-A \in \Phi_+(D(A))$ and $a(\lambda-A)< \infty$ and hence $$\lambda\notin \sigma_{Br_+}(A).$$
\item Suppose that $\int_0^t e^{\lambda(t-s)}\frac{s^{\alpha-1}}{\Gamma(\alpha)}ds\notin\sigma_{Br_-}(S(t))$, then there exist $n,m\in\mathds{N}$ such that
$\beta\left(\int_0^t e^{\lambda(t-s)}\frac{s^{\alpha-1}}{\Gamma(\alpha)}ds-S(t)\right)=m$ and $d\left(\int_0^t e^{\lambda(t-s)}\frac{s^{\alpha-1}}{\Gamma(\alpha)}ds-S(t)\right)=n.$
By Lemma \ref{l1} and Proposition \ref{p2}, we obtain
$\beta(\lambda-A)\leq m$ and $d(\lambda-A)\leq n.$\\
Therefore $\lambda-A \in \Phi_-(D(A))$ and $d(\lambda-A)< \infty$ and hence $$\lambda\notin\sigma_{Br_-}(A).$$
\item It is automatic by the previous assertions of this theorem.
\end{enumerate}
\end{proof}

\begin{prop}\label{p6} Let $A$ be the generator of an $\alpha$-times integrated semigroup $(S(t))_{t\geq 0}$ with $\alpha>0$. Then for all $\lambda\in\mathds{C}$ and all $t\geq 0$, we have
\begin{enumerate}
\item $d_e\left(\int_0^te^{\lambda (t-s)}\frac{s^{\alpha-1}}{\Gamma(\alpha)}ds-S(t)\right)=n,$ then  $d_e(A-\lambda)\leq n;$
\item $a_e\left(\int_0^te^{\lambda (t-s)}\frac{s^{\alpha-1}}{\Gamma(\alpha)}ds-S(t)\right)=n,$ then  $a_e(A-\lambda)\leq n.$
\end{enumerate}
\end{prop}

\begin{proof}
\begin{enumerate}
\item Suppose that $$d_e\left(\int_0^te^{\lambda (t-s)}\frac{s^{\alpha-1}}{\Gamma(\alpha)}ds-S(t)\right)=n.$$
   Since $$R\left(\int_0^te^{\lambda (t-s)}\frac{s^{\alpha-1}}{\Gamma(\alpha)}ds-S(t)\right)^n\subseteq R(\lambda-A)^n,$$
   we define the linear surjective application $\phi$ by
    \begin{eqnarray*}
    \phi: R(\lambda-A)^n &\rightarrow& R\left(\int_0^te^{\lambda (t-s)}\frac{s^{\alpha-1}}{\Gamma(\alpha)}ds-S(t)\right)^n/ R\left(\int_0^te^{\lambda (t-s)}\frac{s^{\alpha-1}}{\Gamma(\alpha)}ds-S(t)\right)^{n+1},\\
     y=(\lambda-A)^nx &\rightarrow & \left(\int_0^te^{\lambda (t-s)}\frac{s^{\alpha-1}}{\Gamma(\alpha)}ds-S(t)\right)^nx +R\left(\int_0^te^{\lambda (t-s)}\frac{s^{\alpha-1}}{\Gamma(\alpha)}ds-S(t)\right)^{n+1}.
    \end{eqnarray*}
Thus, by isomorphism Theorem, we obtain
$$R(\lambda-A)^n/ N(\phi)\simeq R\left(\int_0^te^{\lambda (t-s)}\frac{s^{\alpha-1}}{\Gamma(\alpha)}ds-S(t)\right)^n/ R\left(\int_0^te^{\lambda (t-s)}\frac{s^{\alpha-1}}{\Gamma(\alpha)}ds-S(t)\right)^{n+1}.$$
Therefore $$dim(R(\lambda-A)^n/ N(\phi))= d_e\left(\int_0^te^{\lambda (t-s)}\frac{s^{\alpha-1}}{\Gamma(\alpha)}ds-S(t)\right)=n.$$
Since
$$N(\phi) \subseteq  R\left(\int_0^te^{\lambda (t-s)}\frac{s^{\alpha-1}}{\Gamma(\alpha)}ds-S(t)\right)^{n+1}\subseteq R(\lambda-A)^{n+1},$$
hence $$R(\lambda-A)^n/ R(\lambda-A)^{n+1}\subseteq R(\lambda-A)^n/ N(\phi).$$
Finally, we obtain
$$d_e(\lambda-A)=\dim(R(\lambda-A)^n/ R(\lambda-A)^{n+1} )\leq \dim(R(\lambda-A)^n/ N(\phi))= n.$$
\item Suppose that $$a_e\left(\int_0^te^{\lambda (t-s)}\frac{s^{\alpha-1}}{\Gamma(\alpha)}ds-S(t)\right)=n.$$
   Since $$N(\lambda-A)^{n+1}\subseteq N\left(\int_0^te^{\lambda (t-s)}\frac{s^{\alpha-1}}{\Gamma(\alpha)}ds-S(t)\right)^{n+1},$$
   we define the linear application $\psi$ by
    \begin{eqnarray*}
    \psi: N(\lambda-A)^{n+1} &\rightarrow& N\left(\int_0^te^{\lambda (t-s)}\frac{s^{\alpha-1}}{\Gamma(\alpha)}ds-S(t)\right)^{n+1}/ N\left(\int_0^te^{\lambda (t-s)}\frac{s^\alpha}{\Gamma(\alpha+1)}ds-S(t)\right)^n,\\
     x &\rightarrow & x +N\left(\int_0^te^{\lambda (t-s)}\frac{s^{\alpha-1}}{\Gamma(\alpha)}ds-S(t)\right)^n.
    \end{eqnarray*}
Thus, by isomorphism Theorem, we obtain
$$N(\lambda-A)^{n+1}/ N(\psi)\simeq R(\psi)\subseteq N\left(\int_0^te^{\lambda (t-s)}\frac{s^{\alpha-1}}{\Gamma(\alpha)}ds-S(t)\right)^{n+1}/ N\left(\int_0^te^{\lambda (t-s)}\frac{s^{\alpha-1}}{\Gamma(\alpha)}ds-S(t)\right)^n.$$
Therefore $$\dim N(\lambda-A)^{n+1}/ N(\psi)\leq a_e\left(\int_0^te^{\lambda (t-s)}\frac{s^{\alpha-1}}{\Gamma(\alpha)}ds-S(t)\right)=n.$$
Since
$$N(\psi) \subseteq  N\left(\int_0^te^{\lambda (t-s)}\frac{s^{\alpha-1}}{\Gamma(\alpha)}ds-S(t)\right)^n\subseteq R(\lambda-A)^n,$$
hence $$N(\lambda-A)^{n+1}/ N(\lambda-A)^n\subseteq N(\lambda-A)^{n+1}/ N(\psi).$$
Finally, we obtain $$a_e(\lambda-A)=\dim N(\lambda-A)^{n+1}/ N(\lambda-A)^n\leq \dim N(\lambda-A)^{n+1}/ N(\psi)\leq n .$$
\end{enumerate}
\end{proof}

We will discuss in the following result the essential ascent and descent spectrum.
\begin{thm} Let $A$ be the generator of an $\alpha$-times integrated semigroup $(S(t))_{t\geq 0}$ with $\alpha>0$. Then for all $t\geq 0$
\begin{enumerate}
\item $\int_0^t e^{(t-s)\sigma_{a_e}(A)}\frac{s^{\alpha-1}}{\Gamma(\alpha)}ds\subseteq \sigma_{a_e}(S(t));$
\item $\int_0^t e^{(t-s)\sigma_{d_e}(A)}\frac{s^{\alpha-1}}{\Gamma(\alpha)}ds\subseteq \sigma_{d_e}(S(t)).$
 \end{enumerate}
\end{thm}

\begin{proof}
\begin{enumerate}
\item Suppose that  $$\int_0^t e^{(t-s)\lambda}\frac{s^{\alpha-1}}{\Gamma(\alpha)}ds\notin \sigma_{a_e}(S(t)).$$
    Then there exists $n\in\mathds{N}$ satisfying $$a_e\left(\int_0^t e^{(t-s)\lambda}\frac{s^{\alpha-1}}{\Gamma(\alpha)}ds-S(t)\right)=n.$$
    Therefore, by Proposition \ref{p6}, we obtain
    $a_e(\lambda-A)\leq n$ and hence $$\lambda\notin\sigma_{a_e}(A).$$
\item Suppose that $$\int_0^t e^{(t-s)\lambda}\frac{s^{\alpha-1}}{\Gamma(\alpha)}ds\notin \sigma_{d_e}(S(t)).$$
    Then there exists $n\in\mathds{N}$ satisfying $$d_e\left(\int_0^t e^{(t-s)\lambda}\frac{s^{\alpha-1}}{\Gamma(\alpha)}ds-S(t)\right)=n.$$
    Therefore, by Proposition \ref{p6}, we obtain
    $d_e(\lambda-A)\leq n$ and hence $$\lambda\notin\sigma_{d_e}(A).$$
\end{enumerate}
\end{proof}


{\small
\begin{thebibliography}{99}
\bibitem{r.4}\textsc{P. Aiena,} \emph{Fredholm and Local Spectral Theory with Applications to Multipliers, } Kluwer. Acad. Press, 2004.
\bibitem{r.5}\textsc{W. Arendt,} \emph{Vector-valued Laplace Transforms and Cauchy Problems,} Israel J. Math, 59 (3) (1987), 327-352.
\bibitem{r.6}\textsc{A. Elkoutri and M. A. Taoudi,} \emph{Spectral Inclusions and stability results for strongly continuous semigroups,}
Int. J. of Math. and Mathematical Sciences, 37 (2003), 2379-2387.
\bibitem{r.1}\textsc{M. Heiber,} \emph{Laplace transforms and $\alpha-$times integrated semigroups,} Forum Math. 3 (1991), 595-612.
\bibitem{r.8}\textsc{C. Kaiser,} \emph{Integrated semigroups and linear partial differential equations with delay,} J. Math Anal and Appl. 292 (2) (2004), 328-339.
\bibitem{r.9}\textsc{J.J. Koliha and T.D. Tran,} \emph{ The Drazin inverse for closed linear operators and asymptotic convergence of $C_0$-semigroups,} J.Oper.Theory. 46 (2001), 323–336.
\bibitem{r.2}\textsc{C. Miao Li and W. Quan Zheng,} \emph{$\alpha$-times integrated semigroups: local and global,} Studia Mathematica 154 (3) (2003), 243-252.
\bibitem{r.10}\textsc{V. Müller,} \emph{Spectral theory of linear operators and spectral systems in Banach algebras 2nd edition,} Oper.Theo.Adva.Appl, 139 (2007).
\bibitem{r.11}\textsc{A. Pazy,} \emph{Semigroups of Linear Operators and Applications to Partial Differential Equations,} Applied Mathematical Sciences, Springer-Verlag, New
York 1983.
\bibitem{r.12}\textsc{A. Tajmouati, A. El Bakkali and M.B. Mohamed Ahmed,} \emph{Spectral inclusions between $\alpha$-times integrated semigroups and their generators,} Submitted.
\bibitem{r.3}\textsc{A. Tajmouati and H. Boua,} \emph{Spectral theory for integrated semigroups,} Inter Journal of Pure and Appl Math, 104 (4) (2016), 847-860.
\bibitem{r.7}\textsc{A.E. Taylar and D.C. Lay,} \emph{ Introduction to Functional Analysis,} 2nd ed. New York: John Wiley and Sons, 1980.
\end{thebibliography}}
\end{document}